\documentclass[12pt]{article}
\usepackage{amsmath,amsthm}
\usepackage{amsfonts,amsmath}
\usepackage{amssymb}

\newtheorem{theorem}{Theorem}[section]
\newtheorem{lemma}[theorem]{Lemma}
\newtheorem{corollary}[theorem]{Corollary}

\newtheorem{example}[theorem]{Example}

\renewcommand{\geq}{\geqslant}
\renewcommand{\leq}{\leqslant}

\def\cref#1{Corollary~$\ref{#1}$}

\usepackage{rotating}
\input epsf

\usepackage{amsmath}

\begin{document}

\title{A polynomial embedding of pairs of orthogonal partial latin squares }

\author{Diane M. Donovan\footnote{This work was carried out during a scientific visit to Ko\c{c} University supported by TUBITAK Visiting Scientist program 2221} \footnote{Author was also supported by Australian Research Council (grant number DP1092868)}  \\
dmd@maths.uq.edu.au\\
Centre for Discrete Mathematics and Computing,\\
University of Queensland,
St Lucia 4072 Australia \\
\\
Emine \c{S}ule Yaz{\i}c{\i}\footnote{This work was supported by Scientific and Technological Research Council of Turkey TUBITAK Grant Number:110T692}\\
eyazici@ku.edu.tr\\
Department of Mathematics, Ko\c{c} University, Sar{\i}yer,
34450, \.{I}stanbul, Turkey\\
}

\maketitle

\begin{abstract}
We show that a pair of orthogonal partial latin squares of order $n$ can be embedded in a pair of orthogonal latin squares of order at most $16n^4$ and all orders greater than or equal to $48n^4$. This paper provides the first direct polynomial embedding construction.
\end{abstract}

\section{Introduction and Definitions}

Let $[n]=\{0,1,\dots,n-1\}$ and  $N$ represent a set of $n$ distinct elements. A non-empty subset $P$ of $N\times N\times N$  is said to be a {\em  partial latin square}, of order $n$, if for all $(x_1,x_2,x_3),(y_1,y_2,y_3)\in P$ and for all distinct $i,j,k\in \{1,2,3\}$,
\begin{align*}
x_i=y_i\mbox{ and }x_j=y_j\mbox{ implies }x_k=y_k.
 \end{align*}
We may think of $P$ as an $n\times n$ array where symbol $e$ occurs in cell $(r,c)$, whenever $(r,c,e)\in P$. We say cell $(r,c)$ is {\em empty} in $P$ if,  for all  $e\in N$, $(r,c,e)\notin P$.  The {\em volume} of $P$ is $|P|$.
If $|P|=n^2$, then we say that $P$ is a {\em latin square}, of order $n$. We may use a latin square to define a binary operation; that is, for a latin square $A$, define $A(*)$ to be the {\em quasigroup} where for all $x,y\in N$,
\begin{align*}
x* y=z \mbox{ if and only if }(x,y,z)\in A.
\end{align*}
  Two partial latin squares $P$ and $Q$, of the same order are said to be {\em orthogonal} if they have the same non-empty cells and for all $r_1,c_1,r_2,c_2,x,y\in N$
\begin{align*}
\{(r_1,c_1,x),(r_2,c_2,x)\}\subseteq  P\mbox{ implies }\{(r_1,c_1,y),(r_2,c_2,y)\}\not\subseteq Q.
\end{align*}
\begin{example}
\end{example}
\begin{figure}[h]

\begin{align*}
\begin{array}{|c|c|c|c|}
\hline
  0 & 1 & 2 &  \\ \hline
  2 & 0 & 1 & 3 \\ \hline
  3 &  & 0 &  \\ \hline
   & 2 &  & 1 \\ \hline
\end{array}
&&\begin{array}{|c|c|c|c|}
\hline
  0 & 2 & 1 &  \\ \hline
  3 & 1 & 0 & 2 \\ \hline
  1 &  & 2 &  \\ \hline
   & 0 &  & 3 \\ \hline
\end{array}
\end{align*}
 \caption{A pair of orthogonal partial latin squares of order 4}
\end{figure}

Two partial latin squares $P$ and $Q$ are said to be {\em isotopic}, if $Q$ can be obtained  by reordering the rows, and/or reordering the  columns, and/or   relabeling the symbols  of $P$. We say that a partial latin square $P$ can be {\em embedded} in a latin square $L$ if $P\subseteq L$. A pair of orthogonal partial latin squares $(P_1,P_2)$ is said to be embedded in a pair of orthogonal latin squares $(L_1,L_2)$ if $P_i\subseteq L_i$ for each $i\in{1,2}$.

In 1960 Evans \cite{Evans} proved that a partial latin square of order $n$ can always be embedded in some latin square of order $t$ for every $t\geq 2n$. In the same paper Evans raised the question as to whether a pair of finite partial latin squares which are orthogonal can be embedded in a pair of finite orthogonal latin squares.

It is known, (thanks to a series of papers by many authors, see for example \cite{HZ}) that a pair of orthogonal latin squares of order $n$ can be embedded
in a pair of orthogonal latin squares of order $t$ if $t \geq 3n$, the bound of $3n$ being best possible. Obtaining an analogous result for pairs of orthogonal partial latin squares has proved to be an extremely challenging problem. Lindner \cite{Lindner}  showed that a pair of orthogonal partial latin squares can always be finitely embedded in a pair of orthogonal latin squares, however, there was no known method which obtains an embedding of polynomial order (with respect to the order of the partial arrays).
In \cite{HRW}, Hilton et al. formulate some necessary conditions for a pair of orthogonal
partial latin squares to be embedded in a pair of orthogonal latin squares.
 Jenkins \cite{Jenkins}, considered the less difficult problem of embedding a single partial latin square in a latin square which has an orthogonal mate. His embedding was of order $n^2$.

More generally the study of orthogonal latin squares is a very active  area of combinatorics (see \cite{Handbook}). It has been shown that a set of $n-1$ mutually orthogonal latin squares is equivalent to a projective plane of order $n$, (see \cite{Design} for detailed constructions). So the embedding of mutually orthogonal partial latin squares can be viewed as the embedding of sets of partial lines in finite geometries. Also the embedding of partial latin squares has a strong connection with the embedding of block designs. For example many of the initial embeddings of partial Steiner triple systems used embeddings of partial idempotent latin squares (see for example \cite{6n+3}). It has also been suggested that embeddings of block designs with block size $4$ and embeddings of Kirkman triple systems may  make use of embeddings of pairs of  orthogonal partial latin squares (see \cite{HRW}).

The present paper gives a polynomial embedding result for a pair of orthogonal partial latin squares. Specifically we show that a pair of orthogonal partial latin squares of order $n$ can be embedded in a pair of orthogonal latin squares of order at most $16n^4$.

We preface the discussion of our main result with some necessary definitions.

Let $x_1,x_2$,  $y_1,y_2$  and $z_1,z_2$, be pairs of distinct elements of $N$. The  partial latin square $I=\{(x_1,y_1,z_1),(x_1,y_2,z_2),(x_2,y_1,z_2),(x_2,y_2,z_1)\}$ is termed an {\em intercalate}. If $I$ is a subset of a latin square $A$,  then $(A\setminus I)\cup I^\prime$, where $I^\prime=\{(x_1,y_1,z_2),(x_1,y_2,z_1),(x_2,y_1,z_1),(x_2,y_2,z_2)\}$ is a distinct latin square, of the same order. The partial latin square $I^\prime$ is said to be the {\em disjoint mate} of $I$. The form of the partial latin squares $I$ and $I^\prime$ is displayed below. Here the sideline lists the row labels and the headline lists the column labels.
 \begin{align*}
 \begin{array}{c|cc}
 I&y_1&y_2\\
 \hline
  x_1&z_1&z_2\\
   x_2&z_2&z_1
   \end{array}&&
    \begin{array}{c|cc}
 I^\prime&y_1&y_2\\
 \hline
  x_1&z_2&z_1\\
   x_2&z_1&z_2
   \end{array}
 \end{align*}

 A set $T\subseteq A$, where $A$ is a latin square of order $n$, is said to be a {\em transversal}, if
 \begin{itemize}
 \item[-] $|T|=n$, and
 \item[-] for all distinct $(r_1,c_1,x_1),(r_2,c_2,x_2) \in T$, $r_1\neq r_2$, $c_1\neq c_2$ and $x_1\neq x_2$.
 \end{itemize}

Note that a latin square has an orthogonal mate if and only if it can be partitioned into disjoint transversals.

Let $G$ denote the latin square corresponding to the elementary abelian $2$-group of order $2^m$ on the set $[2^m]$ and $G(\star)$ denote the corresponding quasigroup where the operation is the bitwise addition. Let $0$ denote the identity element and recall that for all elements $x\in [2^m]$, $x\star x=0$.  Since $G(\star)$ is associative and commutative, to reduce unnecessary complexity, we will suppress the  $\star$ notation and rewrite $x\star y$ as $xy$ and in addition we will omit unnecessary brackets.

\section{The basic construction}\label{sc:basic case}
In this section we  show that if $P$ and $Q$ are orthogonal partial latin squares, such that $P$ contains no repeated symbols, then we may embed $P$ and $Q$ in orthogonal latin squares of order  $2^{2m}$, for suitably chosen $m$.

We begin by noting that any partial latin square, of order $n$, can be embedded in a latin square of any order greater than or equal to $2n$ (see \cite{Evans}).

Let $P$ and $Q$ be two orthogonal partial latin squares, of order $n$, volume $v$ and based on the set of symbols $[n]$. Assume  that $P$ satisfies the property
\begin{itemize}
\item[-]  for all $x\in[n]$,
$|\{(r,c,x)\in P\mid 0\leq r,c\leq n-1\}|\leq 1.$
\end{itemize}

Embed $Q$ in a latin square $B$ (quasigroup $B(\circ)$) of order $2^m$, where $m$ is a positive integer such that $2n\leq 2^m$, and embed $P$ in an $2^m\times 2^m$ array $A$ in which each symbol in $[2^{2m}]$ occurs exactly once. We will abuse notation  and write $A(*)$ for the algebraic structure defined by the array $A$ where, for all $r,c\in [2^m]$, $r* c=x$ if and only if $(r,c,x)\in A$ (note $x\in [2^{2m}]$).
We make the following simple observation about $A$.

\begin{lemma}\label{A}
For all $r_1,r_2,c_1,c_2\in [2^m]$, if $r_1*c_1=r_2*c_2$, then $r_1=r_2$ and $c_1=c_2$.
\end{lemma}

\begin{proof} This follows directly from the fact that each symbol in $[2^{2m}]$ occurs in precisely one cell of $A$.
\end{proof}

We will construct arrays  ${\cal A}$ and ${\cal B}$, of order $2^{2m}$, which have the following properties:
 \begin{itemize}
 \item[-] The row and columns of ${\cal A}$ and ${\cal B}$ are
indexed by the ordered pairs $(x,y)\in [2^m]\times [2^m]$.
 \item[-] The symbols of ${\cal A}$ will be chosen from the set $[2^{2m}]$.
 \item[-] The symbols of ${\cal B}$ will be chosen from the set $[2^m]\times [2^m]$.
 \item[-] The array  ${\cal A}$ can be  partitioned into $2^m\times 2^m$ subarrays. That is, for $p,q\in [2^m]$, the intersection of rows $\{p\}\times [2^m]$ with columns $\{q\}\times [2^m]$ defines a subarray containing an isotopic copy of $A$ where the rows and columns are permuted, but the symbols are left unchanged.
 \item[-]  The array ${\cal B}$ can be  partitioned into $2^m\times 2^m$ subarrays. That is, for each  $p,q\in [2^m]$,
 the intersection of rows $\{p\}\times [2^m]$ with columns $\{q\}\times [2^m]$ defines a subarray containing an isotopic copy of $B$ where the rows,  columns and symbols have been permuted. \end{itemize}

 Define ${\cal A}$ and ${\cal B}$  as
\begin{align}
{\cal A}=&\{((p,r),(q,c),q  r* p  c)
\mid   (p,q,p  q)\in G\wedge (r,c,r* c)\in  A\},\label{eq:A}\\
{\cal B}=&\{((p,r),(q,c),(p  q,p \, (q  r\circ p  c))
  \mid (p,q,p  q)\in G\wedge (r,c,r\circ c)\in  B\}. \label{eq:B}
\end{align}
Recall that the $\star$ notation has been suppressed, so for instance, $p \, (q  r\circ p  c)=p \star ((q\star  r)\circ (p\star  c))$ in $G(\star)$.

 The following lemma verifies that  ${\cal A}$ and ${\cal B}$ are latin squares of order $2^{2m}$  and that ${\cal A}$  can be partitioned into transversals, one transversal for each symbol of ${\cal B}$ .

\begin{lemma} Let $P$ and $Q$ be partial latin squares of order $n$ and volume $v$. Let $m$ be an integer such that $2n\leq 2^m$. Embed $P$ in an $2^{m}\times 2^{m}$ array $A$, where each cell contains a distinct symbol from the set $[2^{2m}]$. Embed $Q$ in a latin square $B$ of order $2^{m}$.
 Let ${\cal A}$ be as defined in \eqref{eq:A}. Then ${\cal A}$ is a latin square of order $2^{2m}$ and can be partitioned into $2^{2m}$ transversals ${\cal T}_{(z,d)}$, where $(z,d)\in [2^{m}]\times [2^{m}]$.
  Let ${\cal B}$ be as defined in \eqref{eq:B}. Then ${\cal B}$ is a latin square of order $2^{2m}$.
\end{lemma}

\begin{proof} It is not hard to see that ${\cal B}$ is the latin square obtained by taking the direct product of $G$ with $B$ and then permuting the rows and columns and relabeling the entries within the subsquares.

 Assume ${\cal A}$ is not a latin square then, without loss of generality, there exists  a row (column), say $(p,r)$, which contains a repeated  symbol. So assume that there exists  columns $(q,c)$ and $(t,w)$, such that $q  r* p  c = t  r* p  w$. But $q  r,p  c,t  r,p  w\in [2^m]$, so  Lemma \ref{A} implies $q  r= t r$  and $p  c = p  w,$ and since  $G( \star)$ is a group, $q=t$ and $c=w$.   Thus $ {\cal A}$ is a latin square of order $[2^{2m}]$.

Let $(z,d)\in [2^{m}]\times [2^{m}]$ and consider
\begin{align}
{\cal T}_{(z,d)}=\{((p,r),(q,c),q  r* p  c)\in {\cal A}   \mid p  q=z, p\,  (q  r\circ p  c)=d \}.\label{eq:tran}
\end{align}
  For each $(z,d)\in [2^m]\times [2^m]$ we will show that ${\cal T}_{(z,d)}$ is a transversal.

  Let $((p,r),(q,c),q  r* p  c),((s,u),(t,w),t  u* s  w)$ be distinct ordered triples of $ {\cal T}_{(z,d)}$, so $p  q=z=s  t$ and $p  (q  r\circ p  c)=d=s  (t  u\circ s  w)$. We wish to show that, respectively, the first coordinates, the second coordinates and the third coordinates are not equal.

  Assume that this is not the case, so for instance assume $(p,r)=(s,u)$. Since $G(\star )$ corresponds to the elementary abelian $2$-group and $B(\circ)$ is a quasigroup, if
$p= s$ and $p  q=s  t$  then $q= t$, and if $r=u$   and  $p(q  r\circ p  c)=s(t  u\circ s  w)$, then
$q  r\circ p  c=q  r\circ p  w$ implying $c=w$. Thus $ {\cal T}_{(z,d)}$ contains at most one cell from each row and  similarly at most one cell from each column. In addition, for each $p\in [2^m]$ there exists a $q\in [2^m]$ such that $pq=z$, and given such a pair $p,q,$ for each $r\in [2^m]$ there exists a $c\in [2^m]$ such that $p\,  (q  r\circ p  c)=d$. So we may also deduce that  $|{\cal T}_{(z,d)}|=2^{2m}$.

We are left to consider the case where $(p,r)\neq (s,u)$, $(q,c)\neq (t,w)$, but $q  r* p  c=t  u* s  w$.

Since
$q  r,p  c,t  u,s  w\in [2^m]$, Lemma \ref{A} implies $q  r=t  u$ and $p  c=s  w$ and since $B(\circ)$ is a quasigroup, $q  r\circ p  c= t  u\circ s  w.$

But we have  assumed that $p \, (q  r\circ p  c)=d=s  (t  u\circ s  w)$ and so $p= s$.  Since $p  c=s  w$, we obtain $c=w$. Also $pq=z=st=pt$ so $q=t$. But $qr=tu$, so $r=u$.

Thus the collection ${\cal T}_{(z,d)}$, $(z,d)\in [2^m]\times [2^m]$, partitions ${\cal A}$ into disjoint transversals.

\end{proof}

\begin{theorem}\label{basic case} Suppose $2^{m}\geq2n$. Let $P$ and $Q$ be a pair of orthogonal partial latin squares of order $n$, such that each symbol of $[n]$
occurs in at most one cell of $P$. Then $P$ and $Q$ can be embedded in orthogonal latin squares ${\cal A}$ and ${\cal B}$ of order
$2^{2m}$.
\end{theorem}

\begin{proof} The above construction gives latin squares    ${\cal A}$ and ${\cal B}$ satisfying \eqref{eq:A} and \eqref{eq:B}. Then for all $(z,d)\in [2^{m}]\times [2^{m}]$, take  ${\cal T}_{(z,d)}$ as defined in \eqref{eq:tran}. Note that for fixed $(z,d)$ the cells in ${\cal B}$ defined by ${\cal T}_{(z,d)}$  all contain the symbol
$(z,d)$ in ${\cal B}$. The result is now immediate.
\end{proof}

\section{Orthogonal trades in $A$}\label{sc:trades}

In the next section we would like to vary the above construction to allow for repeated symbols in the partial latin square $P$. We will do this by showing that ${\cal A}$ contains carefully selected intercalates which can be removed and replaced by disjoint mates to obtain a latin square with repeated elements in the subsquare defined by the intersection of rows $(0,r)$, $0\leq r\leq 2^m-1$ and columns $(0,c)$, $0\leq c\leq 2^m-1$ (termed the top left corner), such that the new latin square is also orthogonal to ${\cal B}$.

In the previous section, in ${\cal A}$ the cell defined by row $(0,r_2)$ and column $(0,c_2)$ contained the symbol $r_2*c_2$ and we want to be able to replace it by the symbol $r_1 *c_1$ which also occurs in the cell defined by row $(0,r_1)$ and column $(0,c_1)$. In this way we may introduce repeats into the top left corner in ${\cal A}$. Since ${\cal A}$ is a latin square the repeated symbol $r_1 *c_1$  can not occur in the same row or column so we will require $r_1\neq r_2$ and $c_1\neq c_2$. Furthermore, as the new square should also be orthogonal to ${\cal B}$  we will require the corresponding entries in ${\cal B}$ are different, so $r_1\circ c_1\neq r_2\circ c_2$.

  In the next two lemmas we show that this switch is always possible provided $r_1c_1\neq r_2c_2$ in the elementary abelian 2-group.

We formalize these conditions as follows:

For positive integers $r_1,r_2,r_3,c_1,c_2,c_3,x\in [2^m]$, where $(r_1,c_1,x),(r_2,c_2,x)\in P$ and where appropriate $(r_3, c_3,x)\in P$, define Conditions C1 to C4 as follows.
 \begin{itemize}
 \item[C1.] $r_1< r_2<r_3$.
 \item[C2.] $c_1,c_2, c_3$ are all distinct.
 \item[C3.] $r_1\circ c_1, r_2\circ c_2, r_3\circ c_3$ are all distinct.
 \item[C4.] $r_ir_j\neq c_ic_k$ whenever $i\neq j$, $i\neq k$ and $i,j,k\in \{1,2,3\}$.

      \end{itemize}

\begin{lemma}\label{2intercalate-A}  Let $r_1,r_2,c_1,c_2\in [2^m]$ and $(r_1,c_1,x),\ (r_2,c_2,x)$ be a pair of triples in $P$ which satisfy Conditions $C1$, $C2$, $C3$, and $C4$.
   Then ${\cal A}$  contains the following two intercalates, the union of which  will be termed $I_A(r_1,c_1;r_2,c_2)$.

\begin{footnotesize}
\begin{align*}
\begin{array}{c}
\\
\begin{array}{r|ccccccccccc}
&&&&&(r_1r_2(r_1\circ c_1)(r_2\circ c_2),&((r_1\circ c_1)(r_2\circ c_2), & \\
I_A(r_1,c_1;r_2,c_2)\subset{\cal A}                     &(0,c_2) &&(r_1r_2,c_1)&&c_2(r_1\circ c_1)(r_2\circ c_2))&c_1(r_1\circ c_1)(r_2\circ c_2))\\
\hline
&&&&&&\\
R_1=(0,r_2)&r_2* c_2& & r_1* c_1\\
&&&&&&\\
&&&&&&\\
R_2=(c_1c_2,r_1)&r_1* c_1& & r_2* c_2\\
&&&&&&\\
&&&&&&\\
R_3=(c_1c_2(r_1\circ c_1)(r_2\circ c_2),&&&&&&&&\\
r_2(r_1\circ c_1)(r_2\circ c_2))&&&&&r_1* c_1&  r_2* c_2\\&&\\
&&&&&&\\
R_4=((r_1\circ c_1)(r_2\circ c_2),&&&& & \\
r_1(r_1\circ c_1)(r_2\circ c_2))&&&&&r_2* c_2&  r_1* c_1\\
\end{array}\\
\end{array}
\end{align*}
\end{footnotesize}

Further in ${\cal A}$ these eight entries occur in eight distinct cells, each occurring in a different $2^m\times 2^m$ subsquare which contains
a copy of $A$.
\end{lemma}

\begin{proof} The existence of the eight entries is a straightforward computation following from the definition of ${\cal A}$.

Using the row labels $R_1, R_2, R_3, R_4$ as set out in the above table, the fact that $c_1\neq c_2$ and $r_1\circ c_1\neq r_2\circ c_2$, hence $c_1c_2\neq0$ and $(r_1\circ c_1)(r_2\circ c_2)\neq 0$, implies $R_1\neq R_2$, $R_1\neq R_3$, $R_2\neq R_4$ and $R_3\neq R_4$. As $r_1\neq r_2$, $R_1\neq R_4$, $R_2\neq R_3$.
 Similarly one can show that all the columns are distinct. Thus all four rows and columns are distinct from each other.

 We are left to prove the entries occur in different subsquares. So assume that this is not the case and  two of the entries of $I_A(r_1,c_1;r_2,c_2)$ occur in the same $2^m\times 2^m$ subsquare. Recall, that each such subsquare contains a copy of $A$ and the symbols in the cells of $A$ are all distinct. Also note  that cells are in the same subsquare if and only if row wise they have the same  first coordinates and column wise they have the same first coordinates. Now as $r_1r_2\neq 0$, $c_1c_2\neq 0$ and $r_1\circ c_1\neq r_2 \circ c_2$, the only possibility is that $c_1c_2(r_1\circ c_1)(r_2\circ c_2)=0$ and $r_1r_2(r_1\circ c_1)(r_2\circ c_2)=0$. These equations imply $r_1r_2c_1c_2=0$, which contradicts Condition C4 ($r_1r_2\neq c_1c_2$). Thus all eight entries occur in different $2^m\times 2^m$ subsquares.
\end{proof}
\begin{lemma}\label{2intercalate-B} Let $r_1,r_2,c_1,c_2\in [2^m]$ and $(r_1,c_1,x),\ (r_2,c_2,x)$ be a pair of triples in $P$ which satisfy Conditions $C1$, $C2$, $C3$, and $C4$. Then ${\cal B}$  contains the following eight entries. 

\begin{footnotesize}
\begin{align*}
\begin{array}{c}
\\
\begin{array}{r|ccccccccccc}
&&&&(r_1r_2(r_1\circ c_1)(r_2\circ c_2),&((r_1\circ c_1)(r_2\circ c_2), & \\
 J_B(r_1,c_1;r_2,c_2)\subseteq {\cal B}                       &(0,c_2) &(r_1r_2,c_1)&&c_2(r_1\circ c_1)(r_2\circ c_2))&c_1(r_1\circ c_1)(r_2\circ c_2))\\
\hline
&&&&&&\\
(0,r_2)&(0,r_2\circ c_2)&  (r_1r_2,r_1\circ c_1)\\
&&&&&&\\
&&&&&&\\
(c_1c_2,r_1)&(c_1c_2,&  (r_1r_2c_1c_2,\\
&c_1c_2(r_1\circ c_1))&  c_1c_2(r_2\circ c_2))\\
&&&&&&\\
&&&&&&\\
(c_1c_2(r_1\circ c_1)(r_2\circ c_2),&&&&(r_1r_2c_1c_2,&(c_1c_2,\\
r_2(r_1\circ c_1)(r_2\circ c_2))&&&&c_1c_2(r_2\circ c_2))&c_1c_2(r_1\circ c_1))\\
&&&&&&\\
((r_1\circ c_1)(r_2\circ c_2),&&&& & \\
r_1(r_1\circ c_1)(r_2\circ c_2))&&&&(r_1r_2,r_1\circ c_1)&  (0,r_2\circ c_2)\\
\end{array}\\
\end{array}
\end{align*}
\end{footnotesize}
\end{lemma}

\begin{proof} Using the definition of ${\cal B}$ and the binary  operation $\circ$ on $B$, we can calculate the entries directly.

\end{proof}

\begin{corollary}\label{intercalate-pairs} Construct
${\cal A}$ and ${\cal B}$ as in  \eqref{eq:A} and \eqref{eq:B} and  let $r_1,r_2,c_1,c_2\in [2^m]$ and $(r_1,c_1,x),\ (r_2,c_2,x)$ be a pair of triples in $P$ which satisfy Conditions $C1$, $C2$, $C3$, and $C4$.  Then there exists a latin square ${\cal A}^*$ such that ${\cal A}^*$ is orthogonal to ${\cal B}$, the cell $((0,r_2),(0,c_2))$ of ${\cal A}^*$ contains the symbol $r_1 *c_1$ and ${\cal A}^*$ agrees with ${\cal A}$ everywhere except in the subsquare defined by the intersection of rows $(0,r)$ and columns $(0,c)$, $0\leq r,c\leq 2^m-1$.

\end{corollary}

\begin{proof} By Lemma \ref{2intercalate-A}  we know that
there exist two orthogonal latin squares ${\cal A}$ and ${\cal B}$, for which ${\cal A}$ contains the configuration $I_A(r_1,c_1;r_2,c_2)$ which is the union of two intercalates.  We may replace these intercalates by their disjoint mates to obtain  ${\cal A}^*$. Further by Lemma \ref{2intercalate-B} this trade preserves the orthogonality property and so  the result follows.
\end{proof}

Note that $(r_1,c_1)$ and $(r_2,c_2)$ uniquely determine $I_A(r_1,c_1;r_2,c_2)$, and indeed the next lemma verifies that under certain conditions pairs of these configurations will be non-intersecting.

\begin{lemma}\label{no intersection} Construct
${\cal A}$ and ${\cal B}$ as in (1) and (2) and let $r_1,r_2,r_3,c_1,c_2,c_3\in [2^m]$ and $(r_1, c_1, x), (r_2, c_2, x), (r_3, c_3, x)$ be triples
in $P$ which satisfy conditions $C1$ through $C4$. Then ${\cal A}$ contains $I_A(r_1, c_1; r_2, c_2)$ and $I_A(r_1, c_1; r_3, c_3)$ as
described in Lemma \ref{2intercalate-A} and furthermore $I_A(r_1, c_1; r_2, c_2) \cap I_A(r_1, c_1; r_3, c_3) = \emptyset$.
\end{lemma}

\begin{proof} The existence of $I_A(r_1,c_1;r_2,c_2)$ and $I_A(r_1,c_1;r_3,c_3)$ follows from Lemma \ref{2intercalate-A} and conditions C1 through C4, but we need to check that they are non-intersecting.

 Considering $I_A(r_1,c_1;r_2,c_2)$ and $ I_A(r_1,c_1;r_3,c_3)$. Since $r_1,r_2,r_3$ are all distinct and $c_1,c_2,c_3$ are all distinct,
  rows $(0,r_2),\ (0,r_3), \ (c_1c_2,r_1), \ (c_1c_3,r_1) $ are all distinct. Further since $r_1\circ c_1\neq r_2\circ c_2$,
$r_1\circ c_1\neq r_3\circ c_3$, and $r_2\circ c_2\neq r_3\circ c_3$  and so we may add  rows $((r_1\circ c_1)(r_2\circ c_2),r_1(r_1\circ c_1)(r_2\circ c_2))$   and $((r_1\circ c_1)(r_3\circ c_3),r_1(r_1\circ c_1)(r_3\circ c_3))$ to this collection of distinct rows. Thus we have two remaining rows to check.

Thus  we need to check that the following cases are not possible.
\begin{align}
 (c_1c_3(r_1\circ c_1)(r_3\circ c_3),r_3(r_1\circ c_1)(r_3\circ c_3))&=(0,r_2)\label{eq:first}\\
 (c_1c_2(r_1\circ c_1)(r_2\circ c_2),r_2(r_1\circ c_1)(r_2\circ c_2))&=(0,r_3)\label{eq:second}\\
  (c_1c_3(r_1\circ c_1)(r_3\circ c_3),r_3(r_1\circ c_1)(r_3\circ c_3))&=(c_1c_2,r_1)\label{eq:third}\\
 (c_1c_2(r_1\circ c_1)(r_2\circ c_2),r_2(r_1\circ c_1)(r_2\circ c_2))&=(c_1c_3,r_1)\label{eq:fourth}\\
 (c_1c_3(r_1\circ c_1)(r_3\circ c_3),r_3(r_1\circ c_1)(r_3\circ c_3))&=((r_1\circ c_1)(r_2\circ c_2),r_1(r_1\circ c_1)(r_2\circ c_2))\label{eq:fifth}\\
  (c_1c_2(r_1\circ c_1)(r_2\circ c_2),r_2(r_1\circ c_1)(r_2\circ c_2))&=((r_1\circ c_1)(r_3\circ c_3),r_1(r_1\circ c_1)(r_3\circ c_3)).\label{eq:sixth}\end{align}

If \eqref{eq:first} is true then
\begin{align*}
c_1c_3(r_1\circ c_1)(r_3\circ c_3)&=0\\
r_3(r_1\circ c_1)(r_3\circ c_3)&=r_2.
\end{align*}
But this implies $c_1c_3=r_2r_3$ contradicting Condition C4, thus \eqref{eq:first} is not possible and similarly \eqref{eq:second} is not possible.

If \eqref{eq:third} is true then
\begin{align*}
c_1c_3(r_1\circ c_1)(r_3\circ c_3)&=c_1c_2\\
r_3(r_1\circ c_1)(r_3\circ c_3)&=r_1.
\end{align*}
This implies $r_1r_3=c_2c_3$  which contradicts Condition C4, thus \eqref{eq:third} is not possible, and similarly \eqref{eq:fourth}.
If \eqref{eq:fifth} is true then
\begin{align*}
c_1c_3(r_1\circ c_1)(r_3\circ c_3)&=(r_1\circ c_1)(r_2\circ c_2)\\
r_3(r_1\circ c_1)(r_3\circ c_3)&=r_1(r_1\circ c_1)(r_2\circ c_2).
\end{align*}
But this implies $r_1r_3=c_1c_3$ which once again contradicts Condition C4, thus \eqref{eq:fifth} is not possible, similarly \eqref{eq:sixth} is not possible.

Thus the result follows.

\end{proof}

\section{General embedding}\label{sc:general case}

Let $P$ and $Q$ be two orthogonal partial latin squares, of order $n$, volume $v$ and based on the set of symbols $[n]$. Note that there are no restrictions on $P$.  We  begin by replacing entries in some cells of $P$ to obtain a partial latin square that contains no repeated symbols. We then embed this new partial latin square and finally interchange intercalates to obtain orthogonal latin squares which contain $P$ and $Q$.

For each $x\in [n]$ that appears in $P$, define
\begin{align*}
   u_x=&\mbox{min}\{u\mid (u,w,x)\in P\}.
   \end{align*}
So that $(u_x,w_x,x)\in P$, where $u_x$ is the first row in $P$ that contains $x$. Let $m$ be any integer such that $2^{m}\geq 2n$.

\begin{theorem}\label{thm:cond embed} Let $P$ and $Q$ be a pair of orthogonal partial latin squares of order $n$ and volume $v$, on the set of symbols $[n]$
such that for all $r_1,r_2,c_1,c_2\in [2^m]$ and for each incidence of $(r_1,c_1,x),(r_2,c_2,x)\in P$ Conditions $C1$ to $C4$  are satisfied. Then $P$ and $Q$ can be embedded in orthogonal latin squares ${\cal A^*}$ and ${\cal B}$ of order
$2^{2m}$, where $m$ is an integer such that $2n\leq 2^{m}$.
\end{theorem}

\begin{proof} Construct a new partial latin square $P^*$ as follows. If there exists $(r,c,y)\in P$, place $(r,c,x)$ in $P ^*$, where
 \begin{align*}
 x =
  \begin{cases}
  y, & \text{if } r = u_y \\
   e,      & \text{otherwise, where } e\in\{n,\dots,2^{2m}-1\}\text{ and  $e$ has not previously been used.}
  \end{cases}
\end{align*}
 Thus $P^*$ and $P$ will have the same filled cells, however the cells of $P^*$ will contain distinct entries.

Respectively embed $P^*$ and $Q$, in the arrays $A$ and $B$ of order $2^m$ as described in Section \ref{sc:basic case}. Let $A(*)$ and $B(\circ)$ be the respective algebraic structures. Use Theorem \ref{basic case}  to embed $A$ and $B$ in orthogonal latin square ${\cal A}$ and ${\cal B}$, of order at most
$2^{2m}$.

Define $S = \{(u_y, w_y, r_2, c_2)|(r_2, c_2, y) \in P \wedge r_2 \neq u_y\}$. By Lemma 3.1 ${\cal A}$ contains $I_A(u_y, w_y; r_2, c_2)$ for
each $(u_y, w_y, r_2, c_2)\in S$. We now show that if $(u_y, w_y, r_2, c_2)$, $(u_z, w_z, r_3,$ $c_3)$ $\in S$ then $I_A(u_y, w_y; r_2, c_2)\cap
I_A(u_z, w_z; r_3, c_3) =\emptyset $. If $y\neq z$ then $I_A(u_y, w_y; r_2, c_2)$ and $I_A(u_z, w_z; r_3, c_3)$ contain different symbols and so
do not intersect (recall that $P^*$ contains each symbol at most once). If $y = z$ then $I_A(u_y, w_y; r_2, c_2)$ and
$I_A(u_z, w_z; r_3, c_3)$ do not intersect by Lemma 3.4 because $(u_y, w_y, y)$, $(r_2, c_2, y)$, $(r_3, c_3, y)$ satisfy Conditions C1
through C4. Let ${\cal A}^*$ be the result of replacing each intercalate in $I_A(u_y, w_y; r_2, c_2)$ for some $(u_y, w_y; r_2, c_2)\in S$
with its disjoint mate (we have just shown these intercalates are pairwise disjoint). Note that ${\cal A}^*$ contains
a copy of $P$ and thus by Lemma 3.2 ${\cal A}^*$ is orthogonal to ${\cal B}$.
\end{proof}

The above embedding works provided Conditions C1 to C4 are satisfied for all incidences of $(r_1,c_1,x), (r_2,c_2,x)\in P$. If this is not the case then we begin by expanding the columns of $P$ and $Q$ and relabeling the symbols to ensure that these conditions are satisfied and then we may apply Theorem \ref{thm:cond embed}.

\begin{theorem}\label{thm:final} Let $P$ and $Q$ be a pair of orthogonal partial latin squares of order $n$, on the set of symbols $[n]$. Then $P$ and $Q$ can be embedded in orthogonal latin squares ${\cal A^*_r}$ and ${\cal B_r}$ of order 
$2^{4m}$, where $m$ is a positive integer such that $n\leq 2^{m}$.
\end{theorem}

\begin{proof}
Let $P$ and $Q$ be two orthogonal partial latin squares, of order $n$, volume $v$ and based on the set of symbols $[n]$.  Let $m$ be a positive integer such that $n\leq 2^m$. Construct new partial latin squares $P_0$ and $Q_0$ of order $2^{2m}$ as follows:
\begin{align*}
P_0=&\{(r,c+c\cdot 2^m,e)\mid (r,c,e)\in P\}\\
Q_0=&\{(r,c+c\cdot 2^m,e)\mid (r,c,e)\in Q\}.
\end{align*}
Here $+$ and $\cdot$ are the normal operations of addition and multiplication on the positive integers.
Note that $P_0$ and $Q_0$ are of order $2^{2m}$. Observe that since $n\leq 2^m$, $2n\leq 2^{2m}$ so that $Q_0$ can be completed to a full latin square by Evans' result \cite{Evans}. Further, since $(r_1,c_1,x),(r_2,c_2,x),(r_3,c_3,x)\in P$ and $P$ is a partial latin square Conditions C1 and C2 must be satisfied. Since $Q$ is an orthogonal mate to $P$ Condition C3 must be satisfied. Let $i,j,k\in \{1,2,3\}$, $i\neq j$ and $i\neq k$.     $(r_1,c_1+c_1\cdot2^m,x),(r_2,c_2+c_2\cdot2^m,x)$ and $(r_3,c_3+c_3\cdot2^m,x) \in P_0$, then in the elementary abelian $2$-group $G(\star)$ of order $2^{2m}$, since $r_1,r_2,r_3 \in [2^m]$, $r_i\star r_j$ belongs to a subgroup of order $2^m$, but $c_i+c_i\cdot2^m$ and $c_k+c_k\cdot2^m$ are not elements of this subgroup. Indeed they belong to distinct cosets of the quotient group. Thus
 \begin{align*}
(c_i+c_i\cdot2^m)(c_k+c_k\cdot2^m)\neq r_ir_j.
\end{align*}
and so Condition C4 is satisfied.
Consequently the partial latin squares $P_0$ and $Q_0$ satisfy the conditions of Theorem \ref{thm:cond embed} and  can be embedded in orthogonal latin squares of order $2^{4m}$, validating the proof of this theorem. Note that there exists an isotopic copy of $P$ and $Q$ on the intersection of the rows $\{0\}\times[2^m]$ and columns $\{0\}\times \{c+c \cdot 2^m | c \in[2^m]\}$ in ${\cal A^*}$ and ${\cal B}$, respectively. By reordering the columns of ${\cal A^*}$ and ${\cal B}$ we may obtain the required latin squares ${\cal A^*_r}$ and ${\cal B_r}$.

\end{proof}

\begin{corollary}
Let $P$ and $Q$ be a pair of orthogonal partial latin squares of order $n$, on the set of symbols $[n]$. Then $P$ and $Q$ can be embedded in orthogonal latin squares ${\cal A^*_r}$ and ${\cal B_r}$ of order at most $16n^4$ and all orders greater than or equal to $48n^4$.
\end{corollary}

\begin{proof} By Theorem \ref{thm:final} there is an embedding of order $2^{4m}$ where $m$ is the smallest integer such that $n\leq 2^{m}$. Then $2^{m-1}< n \leq 2^m$ so $2^m < 2n$. Hence $2^{4m} < 16 n^4$. Since a pair of orthogonal latin squares of order $t$ can be embedded in a pair of orthogonal latin squares of all orders greater than or equal to $3t$ \cite{HZ} the result follows.
\end{proof}


\begin{thebibliography}{}

\bibitem{Handbook} C.J. Colbourn and J.H. Dinitz,
Handbook of Combinatorial Designs (Second edition), CRC/Chapman
and Hall, 2006.

\bibitem{Evans}
T. Evans, Embedding incomplete latin squares, {\em Amer. Math. Monthly} \textbf{67} (1960), 958--961.

\bibitem{HZ} K. Heinrich and L. Zhu, Existence of orthogonal latin squares with aligned subsquares,
{\em Discrete Math} \textbf{59} (1986), 69--78.

 \bibitem{HRW}
 A. J. W. Hilton, C. A. Rodger, and J. Wojciechowski, Prospects for good embeddings of pairs of partial orthogonal latin squares and of partial Kirkman triple systems, {\em J Combin. Math. Combin. Comput.} \textbf{11} (1992), 83--91.

\bibitem{Jenkins} P. Jenkins, Embedding a latin square in a pair of orthogonal latin squares, {\em J. Combin. Des.} \textbf{14} (2005), 270--276.
\bibitem{6n+3}
C. C. Lindner, A partial Steiner triple system of order $n$ can be embedded in a Steiner triple system of order $6n+3$,
{\em J. Combin. Theory
Ser. A} \textbf{18}
(1975), 349--351.
\bibitem{Lindner} C. C. Lindner, Embedding orthogonal partial latin squares, {\em Proc. Amer. Math. Soc.} \textbf{59}
(1976), no. 1, 184--186.

\bibitem{Design}
C. C. Lindner, C. A. Rodger, Design Theory, Second Edition, CRC Press, 2009.

\end{thebibliography}
\end{document}